\documentclass[12pt]{amsart}
\usepackage{amsmath,amssymb,enumerate,verbatim}
\usepackage{amsmath, amssymb}
\usepackage[normalem]{ulem}
\usepackage{cancel}
\usepackage{srcltx}

\usepackage[usenames,dvipsnames]{xcolor}
\usepackage{graphicx}
\usepackage{multirow}
\usepackage{booktabs}
\usepackage{mathtools,bm}
\usepackage{blkarray}
\usepackage{young}

\usepackage[all]{xy}

\newcommand\cf{\textit{cf.~}}
\newcommand\ie{\textit{i.e.,~}}
\newcommand\eg{\textit{e.g.~}}

\newcommand\disp{\displaystyle}

\newcommand{\ep}{{\epsilon}}

\newcommand\Ep{\mathcal{E}}

\newcommand\codim{\operatorname{codim}}

\newcommand\al{\alpha}
\newcommand\la{\lambda}

\newcommand\bigb{\ \big|\ }

\newcommand{\bbC}{{\mathbb{C}}}

\newcommand{\bbR}{{\mathbb R}}

\newcommand{\bbZ}{{\mathbb Z}}

\newcommand{\oX}{{\overline{X}}}

\newcommand{\dpfr}{\displaystyle\frac{1}{2} }
\newcommand\one{1\!\!1}

\newcommand{\fk}{\mathfrak}
\newcommand{\ovl}{\overline}
\newcommand{\bb}{\mathbb}

\newcommand{\Ind}{{\mathrm{Ind}}}

\newcommand{\Hom}{{\mathrm{Hom}}}
\newtheorem{definition}[subsection]{Definition}
\newtheorem{lemma}[subsection]{Lemma}

\newtheorem{prop}[subsection]{Proposition}
\newtheorem{theorem}[subsection]{Theorem}
\newtheorem{cor}[subsection]{Corollary}
\newtheorem{remark}[subsection]{Remark}

\newcommand{\calO}{{\mathcal{O}}}
\newcommand{\ccO}{{\overline{\mathcal{O}}}}

\newcommand{\calU}{{\mathcal{U}}}

\newcommand\clrr{\color{red} }

\newcommand{\tu}{\widetilde}
\newcommand{\wtu}{\widetilde}
\newcommand{\wti}{\widetilde}
\newcommand{\wht}{\widehat}

\newcommand\ad{\operatorname{ad}}

\numberwithin{equation}{subsection}

\begin{document}

\title{Representations associated to small nilpotent orbits for complex Spin Groups}

\begin{abstract}
This paper 
provides a comparison between the
$K$-structure of unipotent representations and regular sections of
bundles on nilpotent orbits for complex  groups of type $D$. Precisely,
let $ G_ 0 =Spin(2n,\bbC)$ be the Spin complex group 
viewed as a real group, and $K\cong G_0$ be the complexification of the maximal compact subgroup of $G_0$.
 We compute 
$K$-spectra of the regular functions on some small nilpotent orbits $\calO$ transforming according to characters
$\psi$ of $C_{ K}(\calO)$ trivial on the connected component of the
identity $C_{ K}(\calO)^0$. We then match them with the
${K}$-types of the genuine (\ie representations which do not factor
to $SO(2n,\bb C)$) unipotent  representations attached to $\calO$.  
 
 \end{abstract}

\author{Dan Barbasch}
      \address[D. Barbasch]{Department of Mathematics\\
               Cornell University\\Ithaca, NY 14850, U.S.A.}
        \email{barbasch@math.cornell.edu}
\thanks{D. Barbasch was supported by an  NSA grant}

\author{Wan-Yu Tsai}
\address[Wan-Yu Tsai]{Institute of Mathematics, Academia Sinica, 6F, Astronomy-Mathematics Building, No. 1, Sec. 4, Roosevelt Road, Taipei 10617, TAIWAN}
\email{wytsai@math.sinica.edu.tw}

\maketitle

\section{Introduction}

Let $G_0\subset G$ be the real points of a complex linear reductive algebraic
group $G$ with Lie algebra $\fk g_0$ and 
maximal compact subgroup $K_0$. Let $\fk g_0=\fk k_0+\fk s_0$
be the Cartan decomposition, and $\fk g=\fk k+\fk s$ be the
complexification. Let $K$ be the complexification of $K_0.$

  \begin{definition}
Let $\calO:= K\cdot e\subset \fk s$. We say that an irreducible
admissible representation $\Xi$ is associated to $\calO,$ if $\calO$
occurs with nonzero multiplicity in the associated cycle in the sense
of \cite{V2}.

An irreducible module $\Xi$ of $G_0$ is called unipotent
associated to a nilpotent orbit $\calO\subset \fk s$ and
infinitesimal character $\la_{\calO}$, if it satisfies
\begin{description}
\item[1] It is associated to $\calO$ and its annihilator
  $Ann_{U(\fk g)}\Xi$ is the unique maximal primitive ideal with
  infinitesimal character $\la_{\calO}$,
\item[2] $\Xi$ is unitary.
\end{description}
Denote by $\calU _{G_0}(\calO,\la_{\calO})$  the set of unipotent
representations of $G_0$ associated to $\calO$ and $\la_{\calO}$.
  \end{definition}

\bigskip
{Let $C_K(\calO):= C_K(e)$ denote the centralizer of $e$ in $K$, and 
let $A_K(\calO):=C_K(\calO)/C_K(\calO)^0$ be the component group.} 
Assume that $G_0$ is connected, and a complex
group viewed as a real Lie group. In this case $G\cong G_0\times G_0,$
and  $K\cong G_0$ as complex groups. Furthermore $\fk s\cong\fk g_0$
as complex vector spaces, and the action of $K$ is the adjoint action.  
In this case it is conjectured that there exists an infinitesimal character
$\la_{\calO}$ such that in addition,
\begin{description}
\item[3] There is a 1-1 correspondence $\psi\in
 \wht{ A_K(\calO)}\longleftrightarrow \Xi(\calO,\psi)\in
  \calU_{G_0}(\calO,\la_{\calO})$ satisfying the additional condition 
$$
\Xi(\calO,\psi)\bigb_{K}\cong R(\calO,\psi),
$$
\end{description}
where
\begin{equation}\label{def-reg-fun}
\begin{aligned}
R(\calO, \psi) &= \Ind_{C_{K}(e)} ^{K}  (\psi) \\
&= \{f: K\to V_{\psi}  \mid f(gx) =\psi(x) f(g) \ \forall g\in K, \ x\in C_K (e)\}
\end{aligned}
\end{equation}
is the ring of regular functions on $\calO$ transforming according to
$\psi$. Therefore, $R(\calO,\psi)$ carries a $K$-representation.

Conjectural parameters $\la_\calO$ satisfying the conditions above
are studied in \cite{B}, along with results establishing the validity of 
this conjecture for large classes of nilpotent orbits in the classical
complex groups. Such parameters $\la_\calO$ are available for the
exceptional groups as well, \cite{B} for $F_4$, and to appear elsewhere
for type $E.$  

\medskip
{
This conjecture cannot be valid for all nilpotent orbits in the case of
real groups; the intersection of a complex nilpotent orbit with $\fk
s$ consists of several components.  $R(\calO,\psi)$ can be
the same for different components, whereas  the representations with
associated variety containing a given component have drastically different
$K$-structures.  Examples can be found in \cite{V1}. As explained in
\cite{V1} Chapter 7 and \cite{V2} Theorem 4.11, 
if the codimension of the orbit $\calO$ is $\ge 2,$ then
$\Xi\mid_K=R(\calO,\phi)-Y$  with $\phi$ an algebraic representation,
and $Y$ an $S(\fk g/\fk k)$-module  supported on orbits of strictly
smaller dimension. The orbits $\calO$ under consideration in this
paper have codimension $\ge 2.$ 
Even when $\codim\calO \ge 2,$ (\eg the case of the minimal orbit in
certain real forms of type $D,$) many examples are known where there
are no representations with associated variety $\calO$ or any real
form of its complexification. }

\bigskip
In this paper we investigate this conjecture for \textit{small} orbits
in the complex case by different techniques than in \cite{B}; paper
\cite{BTs} investigates the analogue for the real $Spin$ groups. 
For the condition of \textit{small} we require that 
$$
[\mu : R(\calO, \psi)]\le c_{\calO}
$$ 
\ie that the multiplicity of any $\mu\in \widehat{K}$ be uniformly
bounded. This puts a restriction on $\dim\calO$:

\begin{equation}\label{eq:dim-cond}
\dim \calO \leq \text{rank}(\fk k) + |\Delta ^+(\fk k,\fk t)|,
\end{equation}
where $\fk t\subset \fk k$ is a Cartan subalgebra, and $\Delta^+(\fk
k,\fk t)$ is a positive system. 
The reason for this restriction is as follows. Let $(\Pi,X)$ be an
admissible representation of $G_0$, and  $\mu$ be the highest weight of a
representation $(\pi,V)\in \wht{K}$ which is dominant for
$\Delta^+(\fk k,\fk t)$. Assume that $\dim\Hom_K[\pi,\Pi]\le C$, and
$\Pi$ has associated variety {\cf \cite{V2})}. Then 
$$
\dim\{ v\ :\ v\in X \text{ belongs to an isotypic component with
} ||\mu||\le t\}\le Ct^{|\Delta^+(\fk k,\fk t)|+\dim \fk t}.
$$
The dimension of $(\pi,V)$ grows like $t^{|\Delta^+(\fk k,\fk t) |}$, the
number of representations with highest weight $||\mu||\le t$ 
grows like $t^{\dim\fk t},$ and the
multiplicities are assumed uniformly bounded. On the other hand, considerations
involving primitive ideals imply that the dimension of this set grows
like $t^{\dim   G\cdot e/2}$ with $e\in\calO,$ and half the
dimension of (the complex orbit) $G\cdot e$ is  the
dimension of the ($K$-orbit) $K\cdot e\in\fk s.$ { In the
  case of type $D,$ condition (\ref{eq:dim-cond}) coincides with being
  spherical, see \cite{P}. Since we only deal with characters of
  $C_{ K}(\calO)$, multiplicity $\le 1$ is guaranteed}. 
   
   In the case of the complex groups of type $D_n$, we 
consider $G_0=Spin(2n,\bbC)$ viewed as a real group, and hence $K\cong G_0$ is the complexification of the maximal compact subgroup $K_0=  Spin(2n)$ of $G$.
 In Section 2 we list all small
nilpotent orbits satisfying (\ref{eq:dim-cond}) and describe the (component groups) of their centralizers.
In Section 3, we compute $R(\calO,\psi)$ for each $\calO$ in \ref{ss:orbits} and $\psi\in \widehat{A_{ K}(\calO)}$. In Section 4
we associate to each $\calO$ an infinitesimal character $\la_{\calO}$ by \cite{B}.
 The fact is that $\calO$ is the minimal orbit which can be the associated variety of 
a $(\fk g, K)$-module with infinitesimal character $(\la_L, \la_R)$, with  $\la_L$ and $\la_R$ both conjugate to $\la_{\calO}$. 
We make a complete list of irreducible modules $\oX(\la_L,\la_R)$ (in terms of Langlands  classification) which are attached to $\calO.$
Then we match the $ K$-structure of these representations with $R(\calO,\psi)$.
This demonstrates the conjecture we state in the beginning of the introduction. The following theorem summarizes this.
\begin{theorem}
With notation as above, view $ G_0={Spin}(2n,\bbC)$ as a real group. The $ K$-structure 
of each representations in $\calU_{G_0}(\calO,\la_{\calO})$ is calculated explicitly and matches the 
$K$-structure of the $R(\calO,\psi)$ with $\psi \in \widehat{A_{K}(\calO)}$.That is, there is a 1-1 correspondence $\psi\in
   \widehat { A_{{K}}(\calO)}\longleftrightarrow \Xi(\calO,\psi)\in
  \calU_{ G _0 }(\calO,\la_\calO)$ satisfying 
$$
\Xi(\calO,\psi)\mid_{ K}\cong R(\calO,\psi).
$$
\end{theorem}

For the case $O(2n,\bb C)$  (rather than $Spin(2n,\bb C)$), the
$K$-structure of the representations studied in this paper were
considered earlier in \cite{McG} and \cite{BP1}.

\section{Preliminaries}
\subsection{Nilpotent Orbits} \label{ss:orbits}
The complex nilpotent orbits of type $D_n$ are parametrized by
partitions of $2n$, with even blocks occur with even multiplicities, and with $I, II$ in the \textit{very even} case (see \cite{CM}).
The small nilpotent orbits satisfying (\ref{eq:dim-cond}) are those $\calO$ with $\dim \calO\le n^2$.

\medskip

We list them out as the following four cases:

$$
\begin{aligned}
Case\ 1:\  &n=2p&&\quad \calO=[3\ 2^{n-2} \ 1]&& \dim \calO = n^2 \\     
Case\ 2:\  &n=2p \ \text{ or }  \  2p+1 &&  \quad \calO = [  3\ 2^{2k} \ 1^{2n-4k-3}]  \  \ \footnotesize{0\le k \le p -1} &&\dim \calO =4nk-4k^2+4n-8k-4\\     
Case\ 3:\  &a=2p &&\quad\calO=[2^n]_{I,II} &&  \dim \calO = n^2-n \\     
Case\ 4:\  &a=2p  \ \text{ or } \  2p+1 &&\quad \calO=[ 2^{2k} \ 1^{2n-4k}] \ \ 0\le k < n/2&& \dim \calO = 4nk-4k^2-2k\\
\end{aligned} 
$$

Note that  these are the orbits listed in \cite{McG}.  The proof of the next Proposition, and the details about the nature of the component groups, are 
in Section \ref{ss:clifford}.

\begin{prop} (Corollary \ref{c:cgp})
  \begin{description}
  \item[Case 1] If $\calO=[3\ 2^{2p-2}\ 1],$ then $A_{{K}}(\calO)\cong
\bb Z_2\times\bb Z_2$. 
\item[Case 2] If $\calO=[3\ 2^{2k}\ 1^{2n-4k-3}]$ with  $2n-4k-3>1,$ then  $A_{{K}}(\calO)\cong
\bb Z_2.$
\item[Case 3] If  $\calO=[2^{2p}]_{I,II}$, then $A_{{K}} (\calO) \cong\bb Z_2.$ 
\item[Case 4] If $\calO=[2^{2k}\ 1^{2n-4k}]$ with $2k<n,$ then
  $A_{{K}} (\calO) \cong 1.$ 
  \end{description}
In all cases  $C_{K}(\calO)=Z( K)\cdot C_{ K}(\calO)^0.$ 

\end{prop}

\section{Regular Sections} 
We use the notation introduced in Sections 1 and 2. 
We compute the centralizers needed for $R(\calO,\psi)$ in $\fk k$ and in ${K}$. We use the standard roots 
and basis for $\fk{so}(2n,\bbC).$ A basis for the Cartan subalgebra is given by $H(\ep_i)$, the root vectors are 
$X(\pm\ep_i\pm\ep_j)$. Realizations in terms of the Clifford algebra and explicit calculations are in Section \ref{ss:clifford}.

Let $e$ be a  representative the orbit $\calO$, and let $\{e,h,f\}$ be the corresponding
Lie triple.  Let
\begin{itemize}
\item $C_{\fk k} (h)_i$ be the $i$-eigenspace of $ad(h)$ in $\fk k$,
\item  $C_{\fk k} (e)_i$ be the $i$-eigenspace of $ad(h)$ in the
  centralizer of $e$ in $\fk k$,
\item $C_\fk k (h)^+:= \sum \limits _{i>0} C_\fk k (h) _i$, and $C_\fk k (e)^+:= \sum \limits _{i>0} C_\fk k (e) _i.$  
\end{itemize}

\subsection{} We describe the centralizer for $\calO= [ 3 \ 2^{2k}\ 1^{2n-4k-3}]$ in detail.  These are Cases 1 and 2. 
 Representatives for $e$ and $h$ are 
$$
\begin{aligned} 
e&= X(\ep_1-\ep_{2k+2}) +X(\ep_1+\ep_{2k+2}) + \sum \limits_{2\le i\le 2k+1} X(\ep_i +\ep _{k+i})\\
h&= 2 H(\ep_1)+\sum\limits_{2\le i\le 2k+1} H(\ep_i) = H(2, \underset{2k}{\underbrace{1,\dots,1}},
\underset{n-1-2k}{\underbrace{0, \dots, 0}}). 
\end{aligned}
$$

Then
\begin{equation}
\begin{aligned}
C_{\fk k} (h)_0 &= \fk{gl} (1)\times \fk{gl} (2k)\times { \fk{so}(2n-2-4k)}, \\
C_{\fk k} (h)_1 &= Span\{ X(\ep_1 - \ep_i) , \   X(\ep_i \pm \ep _j )   , \ 2\le i\le 2k+1<j\le n\}, \\
C_{\fk k} (h)_2 &= Span \{ X( \ep _1\pm \ep _j), \  X(\ep_i+ \ep _l),  2\le i \neq l \le 2k+1<j \le n \}, \\
C_{\fk k} (h)_3 &= Span \{  X( \ep_1 +\ep _i ), \ 2\le i\le 2k+1\}.
\end{aligned}
\end{equation}

Similarly

\begin{equation}
\begin{aligned}
C_{\fk k}(e)_0 &\cong \fk{sp}(2k)\times{ {\fk{so}(2n-3-4k)}},\\
C_{\fk k} (e)_1 &= Span\{ X( \ep_1 - \ep_i ) - X(\ep _{k+i }
\pm \ep_{2k+2}) , \  X(\ep _1 -\ep _{k+i})- X(\ep _i \pm\ep _{2k+2}), \ 2\le i \le k+1, \\ 
& X({\ep_j }\pm \ep_l ) , \ 2\le j\le 2k+1, \ 2k+3\le l \le n\},\\
C_{\fk k} (e)_2 &= C_{\fk k} (h)_2 , \\
C_{\fk k} (e)_3 &= C_{\fk k} (h)_3.  
\end{aligned}
\end{equation}

We denote by $\chi$ the trivial character of $C_{\fk k}(e)$. A representation of ${K}$ will be denoted by its highest weight:
$$
V=V(a_1,\dots, a_p), \quad a_1\ge \dots \ge |a_p|,
$$
with all $a_i\in \bbZ$ or all $a_i\in \bbZ+1/2$.

We will compute 
\begin{equation}
\label{eq:mreal}
\Hom_{C_{\fk k}(e)}[V^*, \chi]=
\Hom_{C_{\fk k}(e)_0}\left [V^*/(C_{\fk k}(e)^+V^*), \chi \right ]:=
\left (V^*/(C_{\fk k}(e)^+V^*\right )^\chi.
\end{equation} 

\subsection{Case 1.} $n=2p$, $\calO=[3\ 2^{n-2}\ 1]$.

In  this case $C_{\fk k}(h)_0 = \fk{gl}(1)\times \fk{gl} (n-2)\times \fk{so}(2), C_{\fk k}(e)_0 =\fk{sp}(n-2)$.
\medskip

Consider the parabolic $\fk p = \fk l +\fk n$
determined by $h$,
\begin{equation}
\label{eq:ell}
\begin{aligned}
\fk l &= C_{\fk k}(h)_0 \cong \fk{gl}(1)\times \fk{gl}(n-2)\times \fk{so}(2), \\
\fk n &= C_{\fk k}(h)^+.
\end{aligned}
\end{equation}
We denote by $V^*$, the dual of $V$. Since $n=2p,$ $V^*\cong V.$ Then
$V^*$ is a quotient of a generalized Verma module $M(\la)=U(\fk
k)\otimes _{U(\overline{\fk p} )}  F(\la)$, where $\la$ is a weight of
$V^*$ which is dominant for $\overline{\fk p}$. This is 
$$
\la = (-a_1;-a_{n-1} , \dots, -a_2; -a_n ).
$$ 
The $;$ denotes the fact that this is a (highest) weight of $\fk l\cong \fk{gl}(1)\times \fk{gl}(n-2)\times \fk{so}(2)$.

We choose the standard positive root
system $\triangle^+ (\fk l)$ for $\fk l$. { As a $C_\fk k(e)_0$-module,  
$$
\fk n =C_{\fk
  k}(e)^+\oplus \fk n ^{\perp},
$$ 
where we can choose $\fk n ^{\perp}=Span \{  X(\ep_1-\ep_{j} ) , \ 2\le j\le n-1
\}$. This complement is $\fk l$-invariant.
It restricts to the standard
module of $C_{\fk k}(e)_0=\fk{sp}(n-2).$ }

\medskip
The  generalized Bernstein-Gelfand-Gelfand resolution is:
\begin{equation}
  \label{eq:bgg:cx}
0 \dots \longrightarrow \bigoplus _{w\in W^+, \ \ell(w)=k}M(w\cdot\la) \longrightarrow\dots   \longrightarrow  \bigoplus _{w\in W^+, \ \ell(w)=1}M(w\cdot\la) \longrightarrow M(\la) \longrightarrow
V^*  \longrightarrow 0,
\end{equation}
with $w\cdot \la:= w(\la+\rho(\fk k))-\rho(\fk k)$, and $w\in W^+$, the $W(\fk l)$-coset representatives that make $w\cdot \la$ dominant for $\Delta^+(\fk l).$ This
is a free $C_{\fk k}(e)^+$-resolution so we can compute cohomology by considering 
\begin{equation}
  \label{eq:coh}
0 \dots \longrightarrow \bigoplus _{w\in W^+, \ \ell(w)=k} \ovl{M(w\cdot\la)} \longrightarrow\dots   \longrightarrow  \bigoplus _{w\in W^+, \ \ell(w)=1}\ovl{M(w\cdot\la)} \longrightarrow \ovl{M(\la)}   \longrightarrow \ovl{V^*}  \longrightarrow  0,
\end{equation} 
where $\ovl{X}$ denotes $X/[C_{\fk k}(e)^+ X]$.

Note that in the sequences, 
$M(w\cdot\la)\cong S(\fk n)\otimes_{\bb C}
F(w\cdot\la)$ and $\ovl{M(w\cdot\la)}\cong S(\fk n^{\perp})\otimes_{\bb C}
F(w\cdot\la)$.
As an $\fk l$-module, $\fk n^{\perp}$ has  highest weight 
$(1;0, \dots,0,-1 ;0)$. Then $S^k(\fk n^\perp)\cong F(k;0,\dots ,0 ,-k;0)$
as an $\fk l$-module. 

Let $\mu:=(-\alpha_1 ; -\alpha_{n-1} , \dots, -\alpha_2 ; -\alpha_n)$ be the highest weight of an $\fk l$-module.
By the Littlewood-Richardson rule,  
\begin{equation}
  \label{eq:tensorproduct}
S^k(\fk n ^{\perp})\otimes F_{\mu}=\sum V(-\alpha_1 +k; -\alpha_{n-1} -k_{n-1} ,
\dots, -\alpha_3-k_3, -\alpha_2-k_2;-\alpha_n ). 
\end{equation}
The sum is taken over 
$$\{k_i\ | \ k_i\ge 0,\ \sum k_i =k, \   0\le k_i\le \alpha_{i-1}-\alpha_{i}, \ 3\le i \le n-1\}.$$
  
\begin{lemma}\label{le:Sn}

Hom$ _{C_{\fk k}(e) _0}[S^k (\fk n ^{\perp}) \otimes F_{\mu}  : \chi] \neq 0$ for every $\mu$.
The multiplicity is 1. 
\begin{proof}
Since $(\fk{gl}(n-2),\fk{sp}(n-2))$ is a hermitian symmetric pair, Helgason's
theorem implies that a composition factor in $S(\fk n^\perp)\otimes
F_{\mu}$ admits  $C_{\fk k}(e)_0$-fixed vectors only if 
$$
-\alpha_{n-1}-k_{n-1}=-\alpha_{n-2}-k_{n-2},\
-\alpha_{n-3}- k_{n-3}=-\alpha_{n-4}-k_{n-4},\dots ,-\alpha_3-k_3=-\alpha_2-k_2.
$$

The conditions ${ 0\le k_i\le \alpha_{i-1}-\alpha_{i} } $ imply
\begin{equation}
\begin{aligned}
         k_{n-2} =0, &\quad k_{n-1} = \alpha_{n-2}-\alpha_{n-1}, \\
         \hspace*{4em} \vdots\\
        k_4=0,& \quad  k_5 =\alpha_4-\alpha_5, \\
        k_2=0, & \quad k_3=\alpha_2-\alpha_3.
        \end{aligned}
\end{equation}        

Therefore, given $\mu$, the weight of the $C_{\fk k}(e)_0$-fixed vector in $S(\fk n ^{\perp})\otimes F_{\mu}$ is 
$$
(-\alpha_1+\alpha_2-\alpha_3+\alpha_4-\alpha_5+\dots+\alpha_{n-2}-\alpha_{n-1} ; -\alpha_{n-2},-\alpha_{n-2}, \dots, -\alpha_2,\alpha_2; -\alpha_n),
$$
and the multiplicity is 1.

\end{proof}
\end{lemma}

\begin{cor}
For every $V(a_1,\dots,a_n)\in \widehat{{K}},$, Hom$_{C_{\fk k}(e)} [V,\chi]= 0$ or 1.  The action of  $\ad h$ is $-2\sum\limits_{1\le i \le p} a_{2i-1}$.
\begin{proof}
The first statement follows from Lemma \ref{le:Sn} and the surjection 
$$
\ovl{M(\la)}\cong S(\fk n^{\perp}) \otimes_{\bbC} F(\la)\longrightarrow \ovl{V^*}\longrightarrow 0.
$$
The action of $\ad h$ is computed from the module 
\begin{eqnarray}\label{eq:fix-vec}
V(-a_1+k; -a_{n-2}, -a_{n-2} , \dots, -a_2,-a_2 ; -a_n)
\end{eqnarray}
with 
$k=a_2-a_3+a_4-a_5+\dots+a_{n-2}-a_{n-1}$. The value is $-2\sum\limits_{1\le i \le p} a_{2i-1}$.
\end{proof}
\end{cor}

\subsubsection*{$\mathbf{\ell(w)=1}$}
To show that the weights in (\ref{eq:fix-vec}) actually occur, it is enough to show that these weights 
do not occur in the term in the BGG resolution 
(\ref{eq:coh}) with $\ell(w)=1$.

 We calculate $w\cdot\la:$
$$
\rho=\rho(\fk k)=(-(n-1) ; -1, -2, \dots, -(n-2); 0)
$$
is dominant for $\ovl{\fk p}$, and 
$$
 \la +\rho = (-a_1-n+1; -a_{n-1}-1, -a_{n-2} -2 , \dots, -a_2-n+2; -a_n). 
 $$  
   There are three elements $w\in W^+$ of length 1. They are the  left
   $W(\fk l)$-cosets of 
$$
w_1=s_{\ep_1- { \ep_{n-1} }}, \ w_2=  s_{ { \ep_{2}}-\ep_n}, w_3=s_{ { \ep
     _{2} } +\ep _n}.
$$ 
So 
\begin{equation}
 \begin{aligned}
 w_1\cdot \la &= { (-a_2 +1 ; -a_{n-1}, -a_{n-2}, \dots, -a_4, -a_3,-a_1-1 ;  -a_n)},\\
 w_2\cdot \la&= { (-a_1; -a_n+1, -a_{n-2}, -a_{n-3}, \dots, -a_3, -a_2; -a_{n-1}-1 )} ,\\
 w_3\cdot\la &= { (-a_1; a_n +1, -a_{n-2},-a_{n-3} ,\dots, -a_3,-a_2;a_{n-1}+1)}. 
\end{aligned}
 \end{equation}

\begin{lemma}
For all $\la$, Hom$_{C_{\fk k}(e)}[\ovl{M(w_i\cdot \la)} , \chi  ]=1$. The eigenvalues of $\ad h$ are different from 
$-2\sum \limits_{1\le i\le p}a_{2i-1}$ for each $w_i$.

\begin{proof}

The $\fk{sp}(n-2)$-fixed weights come from {$S(\fk n^\perp)\otimes
F(w_i\cdot \la), \ i=1, 2, 3, \ $ are }

\begin{equation}\label{eq:spn-2}
{ \begin{aligned}
 w_1&&\longleftrightarrow& \mbox{  \footnotesize $( a_1-a_2-a_3 +a_4 -a_5+\dots +a_{n-2}-a_{n-1} +2 ; -a_{n-2},-a_{n-2}, \dots , -a_4,-a_4, -a_2,-a_2; -a_n)$ },\\
 w_2 &&\longleftrightarrow & \mbox{ \footnotesize $ (-a_1+a_2-a_3+\dots  +a_{n-4}-a_{n-3}+a_{n-2}-a_n+1; -a_{n-2}, -a_{n-2}, \dots,-a_4, -a_4,-a_2,-a_2;-a_{n-1}-1)$ },\\
 w_3 &&\longleftrightarrow& \mbox{ \footnotesize $(-a_1+a_2- a_3+\dots + a_{n-4}-a_{n-3}+a_{n-2}+a_n+1 ; -a_{n-2} , -a_{n-2}, \dots,-a_4,-a_4, -a_2, -a_2; a_{n-1}-1)$}. 
\end{aligned}}
 \end{equation}
 
The negatives of the weights of $h$   are

{
\begin{equation}
  \label{eq:wth}
  \begin{aligned}
&w_0=1&&\longleftrightarrow &&2(a_1+a_3+\dots +a_{n-1}),\\
&w_1&&\longleftrightarrow &&2(a_2+a_3+a_5 \dots +a_{n-1}+1),\\
&w_2&&\longleftrightarrow &&2(a_1+a_3+\dots +a_{n-3} +a_{n}-1),\\
&w_3&&\longleftrightarrow &&2(a_1+a_3+\dots +a_{n-3}-a_n -1).   
  \end{aligned}
\end{equation}
}
The last three weights are not equal to the first one. This completes
the proof.
\end{proof}
\end{lemma}

\bigskip

\begin{theorem}\label{th:reg}
Every representation $V(a_1,\dots,a_n)$ has $C_{\fk k}(e)$ fixed vectors and the multiplicity is 1.
We write $C_K(\calO):= C_K(e)$.
 In summary, 
$$
\text{Ind}_{C_{{K}}(\calO)^0}  ^{{K}} (Triv)=\bigoplus _{a \in \widehat{{K}} } V(a_1,\dots,a_n). 
$$ 

\end{theorem}

Theorem \ref{th:reg} can be
  interpreted as computing regular functions on the universal cover
  $\wti\calO$  of
  $\calO$ transforming trivially under $C_{\fk k}(e)_0$.  
We decompose it further:
\begin{eqnarray}\label{eq:decomp}
R(\tu{\calO}, Triv):= \Ind_{C_{{ K} } (\calO)^0} ^{ K} (Triv) =  \Ind _{C_{{K}} (\calO)} ^{{K}}  \left
  [\Ind_{C_{ {K}} (\calO)^0} ^{  C_{ {K}} (\calO)  } (Triv)
\right ]. 
\end{eqnarray}
The inner induced module splits into
\begin{equation} \label{eq:char}
\Ind_{C_{ {K}} (\calO)^0} ^{  C_{ {K}} (\calO)  } (Triv)=\sum\psi
\end{equation}
where $\psi $ are the irreducible representations of $C_{ K}(\calO)$
trivial on $C_{ K}(\calO)^0.$ Thus, the sum in (\ref{eq:char}) is taken over $\widehat{A_K(\calO)}$.

Then
\begin{equation}\label{sum-reg-fun}
R(\tu{\calO}, Triv) =\text{Ind}_{C_{ K} (\calO)^0} ^{ K}(Triv)=\sum \limits_{\psi \in \widehat{A_K(\calO)}   } R(\calO,\psi).
\end{equation}

We will decompose $R(\calO,\psi)$ explicitly as a representation of
${K}$.

\begin{lemma}

Let $\mu_i$, $1\le i \le 4$, be the following  $K$-types parametrized
by their highest weights:  
\begin{eqnarray*}
&\mu_1= (0, \dots, 0), \mu _2 = (1,0,\dots ,0), \\
&\mu_3= (\frac{1}{2},\dots,\frac{1}{2} ),  \mu_4=
(\frac{1}{2},\dots,\frac{1}{2},-\frac{1}{2} ). 
\end{eqnarray*}
Let  $\psi_i $  be  the restriction of the highest weight of $\mu _i$
to $C_{{K}}(\calO)$, respectively. Then   
\begin{eqnarray*}
\text{Ind}_{C_{{K}}(\calO)^0} ^{  C_{{K}}(\calO)  } (Triv)=\sum \limits_{i=1} ^4 \psi _i.
\end{eqnarray*}

\end{lemma}

\begin{prop} \label{p:regfun1}
The induced representation (\ref{sum-reg-fun}) decomposes as 
$$
\text{Ind}   _{C_{ K} (\calO) }  ^{ K} (Triv) =\sum _{i=1} ^4 R(\calO,\psi_i)
$$
where
\begin{eqnarray*}
 R(\calO, \psi _1)&=  \text{Ind} _{C_{ K} (\calO) }  ^{ K}(\psi_1)=\bigoplus V(a_1,\dots,a_n)&\quad \text{ with } a_i\in \bbZ, \ \sum a_i \in 2\bbZ,   \\
R(\calO, \psi _2)&=  \text{Ind} _{C_{ K} (\calO) }  ^{ K} (\psi_2)=\bigoplus V(a_1,\dots,a_n)&\quad \text{ with } a_i\in \bbZ, \sum a_i \in 2\bbZ+1 , \\
 R(\calO, \psi _3)&= \text{Ind} _{C_{ K} (\calO) }  ^{ K}(\psi_3)=\bigoplus V(a_1,\dots,a_n) &\quad \text{ with } a_i\in \bbZ+1/2,\sum a_i \in 2\bbZ+p,\\
 R(\calO, \psi _4)&= \text{Ind} _{C_{ K} (\calO) }  ^{ K} (\psi_4)=\bigoplus V(a_1,\dots,a_n)&\quad \text{ with } a_i\in \bbZ+1/2, \sum a_i \in 2\bbZ+p+1.
\end{eqnarray*}

\end{prop}

\subsection{Case 2} $\calO=[3\ 2^{2k} \ 1^{2n-4k-3}],$ $0\le k \le p-1$.

\subsubsection*{} Consider the parabolic $\fk p = \fk l +\fk n$
determined by $h$: 
\begin{equation*}
\begin{aligned}
\fk l &= C_{\fk k}(h)_0 \cong \fk{gl}(1)\times \fk{gl}(2k)\times
{\fk{so}(2n-2-4k)}, \\ 
\fk n &= C_{\fk k}(h)^+.
\end{aligned}
\end{equation*}
In this section, let $\ep=-1$ when $n$ is even; $\ep=1$ when $n$ is
odd. The dual of $V,$ denoted  $V^*$, has lowest weight $(\ep
a_n,-a_{n-1},\dots, -a_2, -a_1)$. It is therefore a quotient of a
generalized Verma module 
$M(\la)=U(\fk k)\otimes _{U(\overline{\fk p} )}  F(\la)$, where $\la$
 is dominant for $\overline{\fk p}$, and dominant for the standard
 positive system for $\fk l:$
$$
{\la = (-a_1; \underset{2k}{\underbrace{ -a_{2k+1} , \dots, -a_3,
    -a_2}}; \underset{n-1-2k}{\underbrace{ a_{2k+2}, \dots, a_{n-1}, \ep a_n}}).}
$$ 
$\fk n =C_{\fk k}(e)^+\oplus \fk n ^{\perp}$ as a module for $C_{\fk
  k}(e)_0$. A basis for $\fk n^\perp\subset C_\fk k(h)_1$ is given by 
$$
{ \{ X(\ep_1{ -} \ep _{2k+2}) \}},
\quad
2\le i\le 2k+1.  
$$
This is the standard representation of
$\fk{sp}(2k),$ trivial for {$\fk{so}(2n-4-4k).$} We write {its highest weight as}
$$
{(1;0,\dots ,0,-1;0,\dots ,0)}.
$$
We can now repeat the argument for the case $k=p;$ there is an added
constraint that ${a_{2k+3}}=\dots =a_n=0$ because the representation with highest weight $(a_{2k+2},\dots ,a_{n-1},\ep a_n)$ of $\fk{so}(2n-2-4k)$ must have fixed vectors for ${\fk{so}(2n-3-4k)}.$

Then the next theorem follows.

\begin{theorem}

A representation $V(a_1,\dots,a_n)$ has $C_{\fk k }(e)$ fixed vectors if and only if $$a_{2k+3}=\dots=a_n=0,$$ and the multiplicity is 1. 
In summary, 
$$
\text{Ind}_{C_{ K}(\calO)^0}^{ K} (Triv)=\bigoplus V(a_1,\dots, a_{2k+2},0\dots,0), \quad \text{ with } a_1\ge \dots \ge a_{2k+2}\ge 0, \ a_i\in \bbZ.
$$
\end{theorem} 

As in (\ref{sum-reg-fun}), we decompose $\text{Ind}_{C_{ K}(\calO)^0}^{ K} (Triv)$ further in to sum of $R(\calO,\psi)$ with $\psi\in\widehat{A_K(\calO)}$.

\begin{lemma}

Let $\mu_1, \mu_2 $  be the following  ${K}$-types parametrized by
their highest weights:  
\begin{eqnarray*}
\mu_1= (0, \dots, 0), \mu _2 = (1,0,\dots ,0). 
\end{eqnarray*}
Let  $\psi_i $  be  the restriction of the highest weight of $\mu _i$
to $C_G(\calO)$, respectively. Then   
\begin{eqnarray*}
\text{Ind}_{C_{ K} (\calO)^0} ^{  C_{ K}(\calO)  } (Triv)=\psi_1 +\psi _2.
\end{eqnarray*}

\end{lemma}

\begin{prop}\label{p:regfun2}
The induced representation (\ref{sum-reg-fun}) decomposes as 
$$
\text{Ind}   _{C_{ K} (\calO) ^0}  ^{ K} (Triv) = R(\calO,\psi_1)  + R(\calO,\psi_2) 
$$
where
\begin{eqnarray*}
 R(\calO, \psi _1)&=  \text{Ind} _{C_{K} (\calO) }  ^{ K}(\psi_1)=\bigoplus V(a_1,\dots,a_{2k+2},0,\dots,0)&\quad \text{ with } a_i\in \bbZ, \ \sum a_i \in 2\bbZ,   \\
R(\calO, \psi _2)&=  \text{Ind} _{C_{ K} (\calO) }  ^{ K} (\psi_2)=\bigoplus V(a_1,\dots,a_{2k+2},0,\dots,0)&\quad \text{ with } a_i\in \bbZ, \sum a_i \in 2\bbZ+1. \\
\end{eqnarray*}

\end{prop}

\subsection{}
Now we treat $\calO=[2^{2k} \ 1^{2n-4k}]$ with $0\le k\le p$. These are Cases 3 and 4. When $k=p$ (and hence $n=2p$), the orbit 
is labeled by $I,II$. The computation is similar and easier than the previous two cases. We state the results for $R(\tu{O},Triv)$ as follows.

\begin{theorem}\

\begin{description}
\item[Case 3] For $k=p$, so $n=2p,$ 
$$
\begin{aligned} &\calO _I =[ 2^n]_I,\qquad &R(\tu{\calO_I},Triv)= \text{Ind}_{C_{ K}(\calO_{I}) ^0} ^{ K}(Triv)=\bigoplus V(a_1, a_1, a_3, a_3,
\dots, a_{n-1}, a_{n-1}),\\
&\calO _{II}=[ 2^n]_{II},\qquad &R(\tu{\calO_{II}},Triv)=\text{Ind}_{C_{ K}(\calO_{II}) ^0} ^{ K}(Triv)=\bigoplus V(a_1, a_1, a_3, a_3, \dots,
 a_{n-1}, -a_{n-1}).
 \end{aligned}
 $$
\item[Case 4]  
{For $k\le p-1$}, 
$$
\calO= [2^{2k}\ 1^{2n-4k}],\qquad
R(\tu{\calO},Triv)=\text{Ind}_{C_{ K}(\calO) ^0} ^{ K}(Triv)= {\bigoplus V(a_1, a_1, a_3, a_3, \dots, a_{2k-1}, a_{2k-1}, 0,\dots, 0)},
$$
satisfying $a_1\ge a_3\ge \dots \ge a_{2k-1}\ge 0$.
\end{description}
\begin{proof}
 We treat the case $n=2p$ and ${k\le p-1};$ $n=2p+1$ is similar.  {A representative of $\calO$ is}
 $e=X(\ep_1+\ep_{2})+\dots +X(\ep_{2k-1}+\ep_{2k})$, and the corresponding middle element in the  Lie
 triple is 
$h=H(\underset{2k}{\underbrace{{1,\dots
      ,1}}},\underset{n-2k}{\underbrace{{0,\dots ,0}}})$.  
Thus
\begin{equation}
  \label{eq:ch_i}
  \begin{aligned}
&C_{\fk k}(h)_0=\fk{gl}({2k})\times \fk{so}(2n-{4k})\\
&C_{\fk k}(h)_1=Span\{ X(\ep_i\pm \ep_j)\}\qquad {1\le   i  \le 2k< j\le n},\\
&C_{\fk k}(h)_2=Span\{ X(\ep_l+\ep_m)\}\qquad 1\le l\ne m\le {2k}.     
  \end{aligned}
\end{equation}
and
\begin{equation}
   \begin{aligned}
&C_{\fk k}(e)_0=\fk{sp}(2k)\times \fk{so}(2n-4k)\\
&C_{\fk k}(e)_1=C_{\fk k}(h)_1,\\
&C_{\fk k}(e)_2=C_{\fk k}(h)_2.
  \end{aligned} \label{eq:cei}
\end{equation}
As before, let $\fk p=\fk l+\fk n$ be the parabolic subalgebra
determined by $h,$ and $V=V(a_1,\dots ,a_n)$ be an irreducible
representation of $K$. Since we assumed $n=2p,$  $V=V^*.$ 
In this case $C_{\fk k}(e)^+=\fk n,$  so 
Kostant's theorem implies $V/[C_{\fk k}(e)^+V]=V_{\fk l}(a_1,\dots
a_{2k};a_{2k+1},\dots ,a_n)$ as a $\fk{gl}(2k)\times \fk{so}(2n-4k)$-module. 
Since we want $\fk{sp}(2k)\times
\fk{so}(2n-4k)$-fixed vectors, $a_{2k+1}=\dots =a_n=0,$ and Helgason's
theorem implies $a_1=a_2,a_3=a_4,\dots ,a_{2k-1}=a_{2k}$.

\medskip
{When  $n=2p$, and $\calO=[2^n]_{I,II}$, the calculations are similar to  $k\le p-1.$ The choices $I,II$ are
 $$
 \begin{aligned}
& e_I=X(\ep_1-\ep_2)+X(\ep_3-\ep_4)+\dots+X(\ep_{n-1}-\ep_n)\quad &&h_I=H(1,\dots,1),\\   
&e_{II}=X(\ep_1-\ep_2)+X(\ep_3-\ep_4)+\dots +X(\ep_{n-3}-\ep_{n-2})+X(\ep_{n-1}+\ep_n),\quad
&&h_{II}=H(1,\dots,1, -1).
\end{aligned}
$$  
These orbits are induced from the two nonconjugate maximal parabolic subalgebras with $\fk{gl}(n)$ as Levi  components, and $R(\tu{\calO_{I,II}},Triv)$ are just the induced modules from the trivial representation on the Levi component. }
\end{proof}
\end{theorem}

 We aim at decomposing $R(\tu{\calO},Triv) =\sum R(\calO,\psi)$ with $\psi\in \widehat{A_K(\calO)}$ as before.

\begin{lemma}\
\begin{description}

\item[Case 3] { $n=2p$, $\calO = [2^n]_{I,II}$}. 
Let $\mu_1$, $\mu_2$, $\nu_1$, $\nu_2$, be:  
\begin{eqnarray*}
&\mu_1= (1, \dots, 1), \mu _2 = (\frac{1}{2},\dots\frac12), \\
&\nu_1= (1,\dots,1,-1),  \nu_2=
(\frac{1}{2},\dots,\frac{1}{2},-\frac{1}{2} ). 
\end{eqnarray*}
Let  $\psi_i $  be  the restriction of the highest weight of $\mu _i$
to $C_K(e)$, and $\phi _i$ be the restriction of the highest weight of $\nu_i$, respectively. Then   
\begin{eqnarray*}
\text{Ind}_{C_K (\calO_I)^0} ^{  C_K (\calO_I)  } (Triv)&=&\psi_1 +\psi _2,\\
\text{Ind}_{C_K (\calO_{II})^0} ^{  C_K (\calO_{II})  } (Triv)&=&\phi_1 +\phi _2.
\end{eqnarray*}
The  $\psi_i,\phi_i$ are  viewed as 
representations of $\widehat{A_K(\calO_{I,II})}$, and  $\psi_1$ and $\phi_1$ are $Triv,$ $\psi_2,\phi_2$ are $Sgn.$

\medskip

\item[Case 4] { $\calO=[2^{2k}\ 1^{2n-4k}]$, $0\le k \le p-1$}.
\begin{eqnarray*}
\text{Ind}_{C_K (\calO)^0} ^{  C_K(\calO)  } (Triv)=Triv.
\end{eqnarray*}

\end{description}

\end{lemma}

Then we are able to split up $R(\tu{\calO},Triv)$ as a sum of $R(\calO,\psi)$ as in (\ref{sum-reg-fun}).

\begin{prop} \label{p:regfun3}\

\begin{description}
\item[Case 3] 
{ $n=2p$, $\calO = [2^n]_{I,II}$}:  $R(\tu\calO_{I, II}) = R(\calO_{I,II},Triv) + R(\calO_{I,II}, Sgn)$ with
\begin{eqnarray*}
 R(\calO_I, Triv)&=&  \text{Ind} _{C_{ K}(\calO _I) }  ^{ K} (Triv)=\bigoplus V(a_1,a_1,a_3,a_3,\dots,a_{n-1},a_{n-1}), \quad \text{ with } a_i\in\bbZ,   \\
R(\calO_I, Sgn)&= &\text{Ind} _{C_{ K} (\calO _I) }  ^{ K}(Sgn)=\bigoplus V(a_1,a_1,a_3,a_3,\dots,a_{n-1},a_{n-1}),  \quad \text{ with }  a_i\in\bbZ +1/2, \\
R(\calO_{II}, Triv)&=&\text{Ind} _{C_{ K}(\calO _{II}) }  ^{ K}(Triv)=\bigoplus V(a_1,a_1,a_3,a_3,\dots,a_{n-1},-a_{n-1}),  \quad \text{ with }  a_i\in\bbZ,  \\
R(\calO_{II}, Sgn)&=&\text{Ind} _{C_{ K}(\calO _{II}) }  ^{ K}(Sgn)=\bigoplus V(a_1,a_1,a_3,a_3,\dots,a_{n-1},-a_{n-1}),  \quad \text{ with  } a_i\in\bbZ  +1/2,
\end{eqnarray*}
satisfying $a_1\ge a_3 \ge \dots\ge a_{n-1}\ge 0$.

\bigskip

\item[Case 4]
{ $\calO=[2^{2k}\ 1^{2n-4k}]$, $0\le k \le p-1$}:
 \begin{eqnarray*}
 R(\tu{\calO},Triv)=R(\calO,Triv)= \text{Ind} _{C_{ K} (\calO)} ^{ K} (Triv)= \bigoplus V(a_1,a_1,a_3,a_3,\dots, a_{2k-1},a_{2k-1},0,\dots,0), \  \text{ with } a_i\in \bbZ,
 \end{eqnarray*}
 satisfying $a_1\ge a_3 \ge \dots\ge a_{2k-1}\ge 0$
\end{description}
\end{prop}

\section{Representations with small support}

\subsection{Langlands Classification} Let $G$ be a complex linear
algebraic 
reductive group viewed as a real Lie group. Let $\theta$ be a Cartan
involution with fixed points $K.$ Let $G\supset B=HN\supset H=TA$ be a
Borel subgroup containing a fixed $\theta$-stable Cartan subalgebra $H$, with
$$
\begin{aligned}
&T=\{ h\in H\ \mid \ \theta(h)=h\},\\
&A=\{h\in H\ \mid \ \theta(h)=h^{-1}\}.
\end{aligned}
$$

The Langlands classification is as follows. Let
$\chi\in\widehat H.$ Denote by 
$$
X(\chi):=Ind_B^G[\chi\otimes   \one]_{K\text{-finite}} 
$$
the corresponding admissible standard module (Harish-Chandra
induction). Let $(\mu,\nu)$ be the differentials of $\chi\mid_T$ and
$\chi\mid_A$ respectively. Let $\la_L=(\mu+\nu)/2$ and
$\la_R=(\mu-\nu)/2$. We write $X(\mu,\nu)=X(\la_L,\la_R)=X(\chi).$
\begin{theorem}\ 
  \label{t:langlands}

  \begin{enumerate}
  \item $X(\mu,\nu)$ has a unique irreducible subquotient denoted
    $\ovl{X}(\mu,\nu)$ which contains the $K$-type with extremal
    weight $\mu$ occurring with multiplicity one in $X(\mu,\nu).$
\item $\ovl{X}(\mu,\nu)$ is the unique irreducible quotient when
    $\langle Re\nu,\al\rangle >0$ for all $\alpha\in\Delta(\fk n,\fk h),$ and
    the unique irreducible submodule when $\langle Re\nu,\al\rangle <0$.
  \item $\ovl{X}(\mu,\nu)\cong\ovl{X}(\mu',\nu')$ if and only if there
      is $w\in W$ such that $w\mu=\mu', w\nu=\nu'.$ Similarly for
      $(\la_L,\la_R).$   
  \end{enumerate}
\end{theorem}
Assume $\la_L,\ \la_R$ are both dominant integral. 
Write $F(\la)$ to be the finite dimensional representation of $G$ with infinitesimal character $\la$.
Then
$\oX(\la_L,-\la_R)$ is the finite dimensional representation
$F(\la_L)\otimes {F(-w_0\la_R)}$ where  $w_0\in W$ is the long Weyl
group element. The lowest $K$-type has extremal weight
$\la_L-\la_R$. Weyl's character formula implies
$$
\oX(\la_L,-\la_R)=\sum \limits _{w\in W} \ep(w)X(\la_L,-w\la_R).
$$

\subsubsection*{}
In the following contents in this section, we use different notation as follows. We write
 $(\tu{G}, \tu{K})=(Spin (2n,\bbC), Spin(2n) )$ and  $(G,K)=(SO(2n,\bbC) , SO(2n))$. 

\subsection{Infinitesimal characters}  \label{s:infchar}

From \cite{B}, we can associate to each $\calO$ in Section \ref{ss:orbits} an infinitesimal character
$\la_{\calO}$. The fact is that $\calO$ is the minimal orbit which can be the associated variety of 
a $(\fk g, K)$-module with infinitesimal character $(\la_L, \la_R)$, with  $\la_L$ and $\la_R$ both conjugate to $\la_{\calO}$. 
 The  $\la_{\calO}$ are listed below.

\begin{description}
\item[Case 1]  {$n=2p$, $\calO=[3 \ 2^{n-2} \ 1]$,}   
$$\la_{\calO} =\rho/2= (p-\frac{1}{2}, \dots,\frac{3}{2}, \frac{1}{2}\mid p-1,\dots, 1, 0).$$

\item[Case 2] 
{ $\calO=[3\ 2^{2k} \ 1^{2n-4k-3}]$, $0\le k\le p-1$,} 
$$ \la_{\calO}=(  k+\frac{1}{2}, \dots,\frac{3}{2}, \frac{1}{2}\mid n-k-2,\dots, 1, 0).$$

\item[Case 3]
{$n=2p$, $\calO_{I,II}=[2^n]_{I, II}$,}
\begin{eqnarray*}
\la  _{\calO  _I}&=& \left (  \frac{2n-1}{4} ,\frac{2n-5}{4}  , \dots, \frac{-(2n-7)}{4} ,  \frac{ -(2n-3)}{4}  \right),\\
\la _{\calO_{II}} &=& \left (  \frac{2n-1}{4} ,\frac{2n-5}{4}  , \dots, \frac{-(2n-7)}{4} ,  \frac{ (2n-3)}{4}  \right).
\end{eqnarray*}

\item[Case 4]
{ $\calO=[2^{2k}\ 1^{2n-4k}]$, $0\le k \le p-1$,}
$$
\la_{\calO}= (k,k-1,\dots, 1 ; n-k-1, \dots, 1, 0).
$$
\end{description}

Notice that the infinitesimal characters in Cases 1 and 2 are nonintegral. For instance, in Case 1, $\la_{\calO}=\rho/2$, where $\rho$ is half
sum of the positive roots of type $D_{2p}$. The  integral system is of
type $D_p\times D_p$. The notation $|$  separates the coordinates of  the two $D_p$.

 \subsection{} \label{ss:rep} We define
the following irreducible modules in terms of Langlands classification:

\begin{description}
\item[Case 1] {$n=2p$, $\calO=[3 \ 2^{n-1} \ 1]$}.
\begin{enumerate}
\item[(i)] $\Xi_1 = \oX (\la_{\calO},-\la_{\calO})$;
\item[(ii)] $\Xi_2= \oX (\la_{\calO}, -w_1\la_{\calO} )$, where $w_1\la_{\calO}= (p-\frac{1}{2}, \dots,
\frac{3}{2}, -\frac{1}{2} \mid p-1,\dots, 1, 0)$;
\item[(iii)]  $\Xi_3= \oX (\la_{\calO}, -w_2\la_{\calO})$, where $w_2\la_{\calO}= (p-1,\dots, 1, 0 \mid 
p-\frac{1}{2}, \dots, \frac{3}{2}, \frac{1}{2} )$; 
\item[(iv)]  $\Xi_4= \oX (\la_{\calO},-w_3 \la_{\calO})$, where $w_3\la_{\calO}=(p-1,\dots, 1, 0 \mid 
p-\frac{1}{2}, \dots, \frac{3}{2}, -\frac{1}{2} )$.  
\end{enumerate}

\item[Case 2]
{$\calO=[3\ 2^{2k} \ 1^{2n-4k-3}]$, $0\le k\le p-1$}.
\begin{enumerate}
\item [(i)] $\Xi_1=\oX(\la_{\calO},-\la_{\calO})$;
\item [(ii)]$\Xi_2= \oX(\la_{\calO},-w_1\la_{\calO}),\ w_1\la_{\calO}= (  k+\frac{1}{2}, \dots,\frac{3}{2}, \frac{1}{2}\mid n-k-2,\dots, 1, 0 ).$
\end{enumerate}

\item[Case 3]
{$n=2p$, $\calO_{I,II}=[2^n]_{I, II}$}.
\begin{enumerate}
\item [(i)] $\Xi_{I} =\oX(\la_ {\calO_I} , -\la_{\calO_{I} } )$;
\item [(i$'$)] $\Xi _{I} =\oX(\la _{\calO_{I}}  , -w\la
  _{\calO_{I}})$,\quad  { $w\la _{\calO _I}=\left(\frac{2n-3}{2},\dots ,-\frac{2n-1}{4}\right)$};
\item [(ii)] $\Xi_{II} =\oX(\la_ {\calO_{II}} , -\la_{\calO_{II} } )$;
\item [(ii$'$)] $\Xi '_{II} =\oX(\la _{\calO_{II}},\quad -w\la_{\calO_{II}})$, { $w\la _{\calO
      _{II}}=\left(\frac{2n-3}{4},\dots
      ,-\frac{2n-5}{4},\frac{2n-1}{4}\right)$};
\end{enumerate}
\item[Case 4]  {$\calO=[2^{2k}\ 1^{2n-4k}]$, $0\le k \le p-1$}.
\begin{enumerate}
\item[(i)] $\Xi = \oX(\la_\calO, -\la_\calO) $.
\end{enumerate}
\end{description}

\begin{remark}
The representations introduced above form the set $\calU_{\tu{G}}(\calO,\la_{\calO})$.
\end{remark}

\subsubsection*{Notation} \label{ss:notation}
We write $F(\la)$ for the finite dimensional representation of the appropriate $SO$ or $Spin$ group with infinitesimal character $\la$; 
write $V(\mu)$ for the finite dimensional representation of the appropriate $SO$ or $Spin$ group with highest weight $\mu$.

\subsection{$\tu{K}$-structure} We compute  the $\tu{K}$-types of each representation listed in \ref{ss:rep}.

\subsubsection*{Case 1}

The arguments are refinements of those in
\cite{McG}. 
 Let $\wti{H}$ be the image of $Spin(2p,\bbC)\times Spin(2p,\bb C)$ in $Spin(4p,\bb C)$, and   $\wti{U}$ the image of the maximal compact subgroup $Spin(2p)\times Spin(2p)$ in  $\tu{K}$.  Irreducible representations of $\wti{U}$ can be viewed as $Spin(2p)\times Spin(2p)$-representations such that $\pm(I,I)$ acts trivially. 

Cases (i) and (ii)  factor to representations of $SO(2n,\bbC),$ (iii) and (iv)
are genuine for $Spin(2n,\bbC).$ 
 
The Kazhdan-Lusztig conjectures for
nonintegral infinitesimal character together with Weyl's  formula for the character of a finite dimensional module, imply that 
\begin{equation}
  \label{eq:charfla}
\ovl{X}(\rho/2, -w_i\rho/2)=\sum_{w\in W(D_p\times D_p)} \ep(w) X(\rho/2,-ww_i\rho/2),
\end{equation}
since $W(\la_{\calO})=W(D_p\times D_p)$.

Restricting (\ref{eq:charfla}) to $\wtu{K},$
and using Frobenius reciprocity,  we get
\begin{equation} \label{eq:ind-K-U}
\ovl{X}(\rho/2, - w_i\rho/2)\mid_{\wtu{K}}=Ind_{\tu U }^{\wtu{K}}
[F_1(\rho/2) {\otimes} F_2(-w_i\rho/2)],
\end{equation}
where $F_{1,2}$ are finite dimensional representations of  the two factors $Spin(2p,\bb C)\times Spin(2p,\bb C)$ with infinitesimal 
character $\rho/2$ and $-w_i\rho/2$, respectively. 
The terms $[F_1(\rho/2) {\otimes } F_2(-w_i\rho/2)]$ are 
\begin{description}
\item[(i)] $V (1/2,\dots ,1/2)\otimes V(1/2,\dots ,1/2) \boxtimes   V(0,\dots ,0)\otimes V(0,\dots
  ,0)$,
\item[(ii)] $V(1/2,\dots ,-1/2)\otimes V(1/2,\dots ,1/2) \boxtimes V(0,\dots
  ,0)\otimes V(0,\dots ,0)$,
\item[(iii)] $V(1/2,\dots ,1/2)\otimes V(0,\dots ,0)\boxtimes V(0,\dots
  ,0)\otimes V(1/2,\dots ,1/2)$,
\item[(iv)] $V(1/2,\dots ,1/2)\otimes V(0,\dots ,0)\boxtimes V(0,\dots
  ,0)\otimes V(1/2,\dots ,-1/2)$
\end{description}
as $Spin(n)\times Spin(n)$-representations (see \ref{ss:notation} for the notation).

\begin{lemma} \label{pinrep}
Let  $SPIN_+ =V(\frac{1}{2},\dots,\frac{1}{2}),$ and  $SPIN_- = V(\frac{1}{2},\dots,\frac{1}{2}, -\frac{1}{2}) \in \widehat{Spin(n)}$. Then 
\begin{equation}
\begin{aligned}
SPIN _+ \otimes SPIN_+ &= \bigoplus \limits _{0\le k\le [\frac{p}{2}]}
V(\underset{2k}{\underbrace{1\dots 1}},
\underset{p-2k}{\underbrace{0\dots 0}}  ) , \\
SPIN_+\otimes SPIN_- &=  \bigoplus \limits _{0\le k\le [\frac{p-1}{2}]}
V(\underset{2k+1}{\underbrace{1\dots 1}}.
\underset{p-2k-1}{\underbrace{0\dots 0}}  )
\end{aligned}
\end{equation}
\end{lemma}

\begin{proof}
  The proof is straightforward. 
\end{proof}

 Lemma \ref{pinrep} implies that (\ref{eq:ind-K-U}) becomes

\begin{equation}\label{ind-fine1}
\begin{aligned}
(i)\ &\oX (\rho/2,- \rho/2)\mid_{\wti{K}} &=\text{Ind} ^{\tu{K}} _{\wti U} &\left [\bigoplus \limits _{0\le k \le [\frac{p}{2}] }    V(\underset{2k}{\underbrace{1,\dots,1}},0,\dots,0  )\boxtimes V(0,\dots,0) \right ]
\\ 
(ii)\ &\oX (\rho/2,- w_1\rho/2)\mid_{\wti{K}}  & =\text{Ind} ^{\tu{K}} _{\wti U} &\left [\bigoplus \limits _{0\le k \le [\frac{p-1}{2}] }    V(\underset{2k+1}{\underbrace{1,\dots,1}},0,\dots,0  )\boxtimes V(0,\dots,0) \right ]
\\ 
(iii)\ &\oX (\rho/2, -w_2\rho/2)\mid_{\wti{K}}  &=\text{Ind} ^{\tu{K}}_{\wti U}
&\left [V(1/2,\dots ,1/2)\boxtimes V(1/2,\dots ,1/2)\right] \\
(iv)\ &\oX (\rho/2, -w_3\rho/2)\mid_{\wti{K}}  &=\text{Ind} ^{\tu{K}}_{ \wti U}
&\left [V(1/2,\dots ,1/2)\boxtimes V(1/2,\dots ,-1/2)\right].
\end{aligned}
\end{equation}

\begin{prop}\label{p:n2p}
 \begin{equation}
 \begin{aligned}
&\oX (\rho/2,-\rho/2) |_{\tu K}= \bigoplus
V(a_1,\dots ,a_n),\quad \text{ with } a_i\in\bb Z, \   \sum a_i\in 2 \bb Z, \\
&\oX (\rho/2, -w_1 \rho/2) |_{\tu K} = \bigoplus 
V(a_1,\dots,a_n),\quad \text{ with } a_i\in\bb Z, \   \sum a_i\in 2 \bb Z+1,\\
&\oX (\rho/2,-w_2\rho/2)  |_{\tu K}= \bigoplus
V(a_1,\dots, a_n),\quad \text{ with } a_i\in\bb Z+1/2, \   \sum a_i\in 2 \bb Z+p,\\
&\oX (\rho/2,-w_3\rho/2) |_{\tu K} =\bigoplus
V(a_1,\dots, a_n),\quad   \text{ with } a_i\in\bb Z+1/2, \   \sum a_i\in 2 \bb Z+p+1.
\end{aligned}
\end{equation}
\end{prop}
\begin{proof}
In the first two cases we can substitute $\big({G}^{split},K^{split}):=\big(SO(2p,2p),S[O(2p)\times O(2p)])\big)$ for 
$\big(\wtu{K},\wti U\big),$ and $\big(Spin(2p,2p),Spin(2p)\times Spin(2p)/\{\pm (I,I)\}\big)$   for the last two cases.
The problem of computing
the $\wtu{K}$-structure of $\ovl{X}$ reduces to finding the finite
dimensional representations of $\wtu{G}^{split}$ which contain factors of
$F(\rho/2)\otimes F(-w_i\rho/2).$ 
Any finite dimensional
representation of $\wtu{G}^{split}$ is a Langlands quotient of a principal
series. Principal series have fine lowest $K$-types (see \cite{V}). Let
${M}A$ be a split Cartan subgroup of $\wtu{G}^{split}.$  A principal series is
parametrized by a $(\delta,\nu)\in\wht{M}A.$ The $\delta$ are called
fine, and each fine ${K^{split}}$-type $\mu$ is a direct sum of a Weyl group
orbit of a $fine$ $\delta.$ This implies
that the multiplicities in (\ref{ind-fine1}) are all one, and all the
finite dimensional representations  occur in
$(i),(ii),(iii),(iv)$. The  four formulas correspond to the various orbits of the $\delta.$  
\end{proof}

\subsubsection*{Case 2: $\calO=[3\ 2^{2k}\ 1^{2n-4k-3}]$, $0\le k \le p-1$}

Recall that $$\la_{\calO}= (  k+\frac{1}{2}, \dots,\frac{3}{2}, \frac{1}{2}\mid n-k-2,\dots, 1, 0),$$ and the integral system is $D_k\times D_{n-k}.$
 The irreducible modules
are of the form $\ovl{X}(\la_L,-w\la_R)$ such that $\la_{\calO}$ is
dominant, $w_i\la_{\calO}$ is antidominant for $D_{k}\times D_{n-k},$ and they
factor to $SO(2n,\bb C).$ These representations are listed in \ref{ss:rep}.

\subsubsection*{} We need to work with  the real form $\big(SO(r,s), S[O(r)\times O(s)]\big)$. 
A representation of $O(n)$, $r=2m+\eta$ with $\eta=0$ or $1$, will be denoted by $V(a_1,\dots,a_m; \ep)$,  with $\ep=\pm 1,1/2$ according to Weyl's convention, and $a_1\ge a_2\ge \dots\ge a_m\ge 0.$ If $a_m=0,$ there
are two inequivalent representations with this highest weight, one for $\ep=1$, one for $\ep=-1.$  Each restricts irreducibly to $SO(r)$ as the representation $V(a_1,\dots,a_m)\in \widehat{SO(r)}$. 
 When $a_m\neq 0$, there
is a unique representation with this highest weight, $\ep=1/2$ or
$\ep$ is suppressed altogether. The restriction of this representation to $SO(r)$ is a sum of two representations $V(a_1,\dots,a_m)$ and $V(a_1,\dots,a_{m-1},-a_m)$. 

Representations of $Pin(s)$ are parametrized in the same way, with $a_1\ge \dots \ge a_m\ge 0$ allowed to be nonnegative decreasing half-integers. 

Representations of $S[O(r)\times O(s)]$ are
parametrized by restrictions of $V(a;\ep_1)\boxtimes V(b; \ep_2)$ with
the following equivalences:
\begin{enumerate}
\item If one of $\ep_i=\frac{1}{2}$, say, $\ep_1 =\frac{1}{2}$,
then $V(a;\ep_1)\boxtimes V(b;\ep_2)=V(a'; \delta_1)\boxtimes V(b' ;
\delta_2)$ if and only if $a=a', b=b', \ep _1=\delta _1, \ep_2= \delta_2$. 
\item If $\ep_1, \ep_2, \delta_1, \delta_2\in \{\pm 1\}$, then
$V(a;\ep_1)\boxtimes V(b;\ep_2)=V(a'; \delta_1)\boxtimes V(b' ;
\delta_2)$ iff $a=a', b=b', \ep_1\ep_2=\delta_1\delta_2$. 
\end{enumerate}


\begin{lemma} \label{pin-rep-1}
Let $PIN =V(\frac{1}{2}\dots, \frac{1}{2}) \in \widehat{Pin(s)}, s=2m+\eta$ with $\eta=0$ or 1. Then 
\begin{equation}
PIN\otimes PIN =\sum _{\ell=0} ^{m-1}  V(\underset{k}{\underbrace{1\dots 1}},
\underset{m-\ell}{\underbrace{0\dots 0}} ; \ep ) + V(1,\dots,1;1/2),
\end{equation}
where the sum in over $\ep=1$ and $-1$.
\begin{proof}
Omitted.
\end{proof}
 \end{lemma}
 \subsubsection*{}
 We will use the groups $U=S[O(2k)\times O(2n-2k)]\subset K=SO(2n)$. 
Again, the representations that we want are in \ref{ss:rep}. 
 As before,
\begin{equation}
  \label{eq:cf2}
\ovl{X}(\la_{\calO}, -w_i\la_\calO)=\sum_{w\in W(D_{k}\times D_{n-k})}  {\ep(w)}X(\la_\calO, - ww_i\la_\calO).  
\end{equation}
{Restricting to $K,$
and using Frobenius reciprocity, (\ref{eq:cf2}) implies
\begin{equation} \label{ind-K-U-2}
\ovl{X}(\la_\calO, - w_i\la_\calO)\mid_{K }=\text{Ind}_{U }^{K}
[F_1(\la_{\calO}) {\otimes} F_2(-w_i\la_\calO)].
\end{equation}
}

The terms $[F_1(\la_\calO)\otimes F_2(-w_i\la_\calO)]$ are 
\begin{description}
\item[(i)] $V(1/2,\dots ,1/2)\otimes V(0,\dots ,0)\boxtimes V(1/2,\dots
  ,1/2)\otimes V(0,\dots ,0)$,
\item[(ii)] $V(1/2,\dots ,1/2,-1/2)\otimes V(0,\dots ,0)\boxtimes V(1/2,\dots ,
  1/2,-1/2)\otimes V(0,\dots ,0)$.
\end{description}

\
\begin{lemma}
\begin{equation}\label{ind-fine6}
\begin{aligned}
\oX (\la_\calO, -\la_\calO) =\text{Ind} ^{K} _{U} \Big [ & \sum \limits _{0\le 2\ell \le k }
V(\underset{2\ell}{\underbrace{1,\dots,1}},0,\dots,0 ;1 )\boxtimes
V(0,\dots,0;1)\\ 
\\ &+ \sum \limits _{0\le 2\ell \le k}
V(\underset{2\ell}{\underbrace{1,\dots,1}},0,\dots,0 ;1 )\boxtimes
V(0,\dots,0;-1)\Big],\\ 
\oX (\la_\calO, -w_1\la_\calO) =\text{Ind} ^{K} _U  \Big [ & \sum \limits _{0\le 2\ell+1 \le k }
V(\underset{2l+1}{\underbrace{1,\dots,1}},0,\dots,0 ;1 )\boxtimes
V(0,\dots,0;1)  
\\ &+ \sum \limits _{0\le 2\ell+1 \le k}
V(\underset{2\ell+1}{\underbrace{1,\dots,1}},0,\dots,0 ;1 )\boxtimes
V(0,\dots,0;-1) \Big]. 
\end{aligned}
\end{equation}
\end{lemma}
\begin{proof}
This follows from Lemma  \ref{pin-rep-1}. 
\end{proof}
\begin{prop}
 \begin{equation}\label{eq:ktypeskn}
 \begin{aligned}
& \oX (\la_\calO,-\la_\calO)|_{\tu{K}} = \bigoplus V(a_1,\dots,a_{k}, 0,\dots ,0),\quad  \text{ with } \ a_i\in \bb Z, \ \sum a_i\in 2\bbZ\\
 &\oX (\la_\calO,  -w_1\la_\calO   )|_{\tu{K}} =  \bigoplus  V(a_1,\dots,a_{k},0,\dots ,0),\quad \text{ with } \  a_i\in \bb Z.\ \sum a_i\in 2\bbZ +1.
 \end{aligned}
   \end{equation}  
\end{prop}
\begin{proof}
The proof is almost identical to that of Proposition \ref{p:n2p}. 
When $k=p-1$, 
 the group $\wtu{G^{split}}$ in the proof of Proposition \ref{p:n2p} is replaced by
$G^{qs}=SO(2p,2p+2)$ and $\wtu{U}$ is replaced by $U=S[O(2p)\times O(2p+2)].$  
When $k<p-1$,
the group $\wtu{G^{split}}$ is replaced by
$G^{k,n-k}=SO(2k,2n-2k)$ and $\wtu{U}$ is replaced by $U= S[O(2k)\times O(2n-2k)].$
We follow \cite{V}. 
The $K$-types $\mu$ in (\ref{ind-fine6}) have $\fk q(\la_L)$ 
the $\theta$-stable parabolic $\fk
q=\fk l+\fk u$ determined by $\xi=(0,\dots
,0;\underset{n-2k-2}{\underbrace{1,\dots ,1}},0\dots ,0)$. The Levi
component is $S[O(2k)\times O(2k+2)].$ The resulting
$\mu_L=\mu-2\rho(\fk u\cap\fk s)$  are fine $U\cap L$-types. A bottom layer argument
reduces  the proof to the quasisplit case $n=2p+1$.   
\end{proof}

\subsubsection*{Cases 3,4}
We use the infinitesimal characters in \ref{s:infchar} and the representations are from \ref{ss:rep} again.

In Case 4, $\calO=[2^{2k}\ 1^{2n-4k}]$ with $k<p$. There is a unique irreducible representation with
associated support $\calO$, and it is spherical. It is a special
unipotent representation with character given by \cite{BV}. 

When {$n=2p$ and $k=p$}, there are two nilpotent orbits $\calO_{I,II} = [2^n]_{I,II}$. 
The representations $\Xi_{I,II}$ in \ref{ss:rep} are 
 spherical representations, one each for $\calO_{I,II}$  that are not genuine. 
The two representations are induced irreducibly from the
trivial representation of the parabolic subgroups with Levi components
$GL(n)_{I,II}.$ 
On the other hand, the representations $\Xi'_{I,II}$  are induced irreducibly from the
character $Det ^{1/2}$ of the parabolic subgroups with Levi components
$GL(n)_{I,II}.$ 
All of these are unitary.

\begin{prop}
  \label{p:kstruct3}  
  The $\tu K$-types of these representations are:
  \begin{description}
   \item[Case 3] $\calO_{I,II}=[2^{2p}]_{I,II}:\ $ 
   \begin{equation}
   \begin{aligned}
   \Xi_I |_{\tu{K}} & =& \bigoplus V(a_1,a_1,a_3,a_3,\dots,a_{n-1},a_{n-1}) & \text{ with } a_i\in \bbZ, \\
      \Xi '_I |_{\tu{K}} & =& \bigoplus V(a_1,a_1,a_3,a_3,\dots,a_{n-1},a_{n-1}) & \text{ with } a_i\in \bbZ +1/2,\\
\Xi_{II}|_{\tu{K}} & =& \bigoplus V(a_1,a_1,a_3,a_3,\dots,a_{n-1},-a_{n-1}) & \text{ with } a_i\in \bbZ, \\
      \Xi '_{II}|_{\tu{K}} & =& \bigoplus V(a_1,a_1,a_3,a_3,\dots,a_{n-1},-a_{n-1}) & \text{ with } a_i\in \bbZ +1/2,
   \end{aligned}
   \end{equation}
satisfying $a_1\ge a_3\ge \dots \ge a_{n-1}\ge 0$
\medskip

  \item[Case 4] $\calO=[2^{2k}\ 1^{2n-4k}],\  0\le k< n/2:\ $ 
  $$\Xi | _{\tu{K}}= \bigoplus V(a_1,a_1,\dots,
    a_k,a_k,0,\dots ,0), \ \text{ with } \ a_i\in \bbZ,$$
 satisfying $a_1\ge a_3\ge \dots \ge a_k\ge 0.$
  \end{description}

\begin{proof}
These are well known. 
The cases $[2^{n}]_{I,II}$ follow by Helgason's theorem since
$(D_{n},A_{n-1})$ is a symmetric pair (for the real form
$SO^*(2n)$). They also follow by the method
outlined below for the other cases.  

For $2k<n,$ the methods outlined in \cite{BP2} combined with \cite{B} give the
answer; the representations are $\Theta$-lifts of the trivial
representation of $Sp(2k,\bbC).$ More precisely $\oX(\la_{\calO},-\la_\calO)$
is $\Omega/[\mathfrak{sp}(2k,\bbC)\Omega]$ where $\Omega$ is the oscillator
representation for the pair $O(2n,\bbC)\times Sp(2k,\bbC)$. 
The $K$-structure can then be computed using seesaw
pairs, namely $\Omega$ is also the oscillator representation for the
pair  $O(2n)\otimes Sp(4k,\bbR)$.  
\end{proof}
\end{prop}

\subsection{}
We resume the notation used in Section 3. Let $(G_0, K) = (Spin (2n,\bbC), Spin(2n,\bbC))$.
 By comparing Propositions \ref{p:regfun1}, \ref{p:regfun2}, \ref{p:regfun3}
 and the $K$-structure of representations listed in this section, we have 
 the following matchup.
 
 \begin{description}
\item[Case 1] $\Xi_i |_{ K} = R(\calO,\psi_i), \  1\le i \le 4$;
\item[Case 2] $\Xi_i |_{ K} = R(\calO,\psi_i), \  i=1,2$;
\item[Case 3] $\Xi_I |_{ K} =R(\calO_I, Triv)$,  $\Xi ' _I |_{ K} =R(\calO_I, Sgn)$, \\ \hspace*{2.5em} $\Xi_{II}|_{ K} =R(\calO_{II}, Triv)$, $\Xi '_{II} |_{ K} =R(\calO_{II}, Sgn)$;  
\item[Case 4] $\Xi |_{ K} = R(\calO , Triv)$.
\end{description}

Then the following theorem follows.

\begin{theorem}\label{t:main}
Attain the notation above. Let $ G_0={Spin}(2n,\bbC)$ be viewed as a real group. The $K$-structure 
of each representations in $\calU_{G_0}(\calO,\la_{\calO})$ is calculated explicitly and matches the 
$K$-structure of the $R(\calO,\psi)$ with $\psi \in \widehat{A_{ K}(\calO)}$.That is, there is a 1-1 correspondence $\psi\in
   \widehat { A_{{K}}(\calO)}\longleftrightarrow \Xi(\calO,\psi)\in
  \calU_{ G_0}(\calO,\la_\calO)$ satisfying 
$$
\Xi(\calO,\psi)\mid_{ K}\cong R(\calO,\psi).
$$
\end{theorem}

\section{Clifford algebras and Spin groups} \label{ss:clifford}
Since the main interest is in the case of $Spin(V),$ the simply
connected groups of type $D,$ we realize everything in the context of
the Clifford algebra.  

\subsection{} Let $(V,Q)$ be a quadratic space of even dimension $2n$, with a basis
$\{e_i,f_i\}$ with $1\le i\le n,$ satisfying $Q(e_i,f_j)=\delta_{ij},$
$Q(e_i,e_j)=Q(f_i,f_j)=0$. Occasionally we will replace $e_j,f_j$ by
two orthogonal vectors $v_j,w_j$ satisfying  
$Q(v_j,v_j)=Q(w_j,w_j)=1,$ and orthogonal to the $e_i,f_i$ for $i\ne
j.$  Precisely they will satisfy $v_j=(e_j+f_j)/\sqrt{2}$ and 
$w_j=(e_j-f_j)/(i\sqrt{2})$ (where $i:=\sqrt{-1},$ not an index). Let
$C(V)$ be the Clifford algebra with automorphisms $\al$ defined by 
$\al(x_1\cdots x_r)=(-1)^r x_1\cdots x_r$ and $\star$ given by $(x_1\cdots
x_r)^\star=(-1)^r x_r\cdots x_1, $ subject to the relation
$xy+yx=2Q(x,y)$ for $x,y\in V$. The double cover of $O(V)$ is  
$$
Pin(V):=\{ x\in C(V)\ \mid\ x\cdot x^\star=1,\ \al(x)Vx^\star\subset V\}.
$$
The double cover $Spin(V)$ of $SO(V)$ is given by the elements in $Pin(V)$ which are in $C(V)^{even},$ \ie  $\disp{Spin(V):=Pin(V)\cap C(V)^{even}}.$
For $Spin,$ $\al$ can be suppressed from the notation since it is the identity.

\smallskip
The action of $Pin(V)$ on $V$ is given by $\rho(x)v=\al(x)vx^*.$ The
element $-I\in SO(V)$ is covered by 
\begin{equation}
\label{eq:-spin}
\pm \Ep_{2n}=\pm i^{n-1}vw\prod_{1\le j\le n-1} [1-e_jf_j]=\pm
i^n\prod_{1\le j\le n} [1-e_jf_j].  
\end{equation}
These elements satisfy 
$$
\Ep_{2n}^2=
\begin{cases}
  +Id &\text{ if } n\in 2\bb Z,\\
  -Id&\text { otherwise.}
\end{cases}
$$
The center of $Spin(V)$ is 
$$
Z(Spin(V))=\{\pm I, \pm\Ep_{2n}\}\cong
\begin{cases}
\bb Z_2\times \bb
Z_2 &\text{ if }n \text{ is even,}\\ 
\bb Z_4 &\text{ if } n \text{ is odd}.
\end{cases}
$$
The Lie algebra of $Pin(V)$ as well as $Spin(V)$ is formed of elements
of even order $\le 2$ satisfying
$$
x+x^\star=0.
$$ 
The adjoint action is {$\ad x(y)=xy-yx$}. A
Cartan subalgebra and the root vectors corresponding to the usual
basis in Weyl normal form are  formed of the elements 

\begin{equation}
\begin{aligned}
\label{eq:liea}
&(1-e_if_i)/2&&\longleftrightarrow &&H(\ep_i)\\
&e_ie_j/2&&\longleftrightarrow &&X(-\ep_i-\ep_j),\\
&e_if_j/2&&\longleftrightarrow &&X(-\ep_i+\ep_j),\\
&{f_i} f_j/2&&\longleftrightarrow &&X(\ep_i+\ep_j).
\end{aligned}
\end{equation}

\subsection{Nilpotent Orbits} 
 We write $\tu{K}=Spin(V)=Spin (2n,\bbC) $, $K=SO(V)=SO(2n,\bbC)$.  A nilpotent orbit of an element $e$  will have Jordan blocks denoted by
\begin{equation}
\label{eq:blocks}
\begin{aligned}
&e_1\longrightarrow e_2\longrightarrow\dots \longrightarrow
e_k\longrightarrow v\longrightarrow -f_k\longrightarrow
f_{k-1}\longrightarrow {-f_{k-2} \longrightarrow}  \dots \longrightarrow \pm f_1\longrightarrow 0\\
&\begin{matrix}
&e_1\longrightarrow &e_2&\longrightarrow&\dots &\longrightarrow
&e_{2\ell}\longrightarrow 0\\
&f_{2\ell}\longrightarrow &-f_{2\ell-1}&\longrightarrow&\dots &\longrightarrow
&-f_1\longrightarrow 0
\end{matrix} 
\end{aligned} 
\end{equation}
with the conventions about the $e_i,f_j,v$ as before. There is an even
number of odd sized blocks, and any two blocks of equal odd size $2k+1$
can be replaced by a pair of   blocks of the form as the even ones.  A realization of the odd block is given by $ \displaystyle\frac{1}{2}\left ( {\sum \limits _{i=1} ^{k-1} }e_{i+1}f_i +vf_k\right ),$ and a realization of the even blocks by $\dpfr \left( {\sum \limits _{i} ^{2l-1}}e_{i+1}f_{i}\right).$ When there are only even blocks, there are two orbits; one block of the form 
$\big(\sum_{1\le i< \ell-1} e_{i+1}f_{i}+e_\ell f _{\ell-1}\big)/2$ is replaced by  $\big(\sum_{1\le i< \ell-1} e_{i+1}f_{i}+ f_\ell f_{\ell-1}\big)/2.$ 

\medskip
The centralizer of $e$ in $\fk{so}(V)$  has Levi component isomorphic to a product 
of $\fk{so}(r_{2k+1})$ and $\fk{sp}(2r_{2\ell})$ where $r_j$ is the number of
blocks of size $j.$ The centralizer of $e$ in $SO(V)$ has Levi
component $\prod Sp(2r_{2\ell})\times S[\prod O(r_{2k+1})]$. 
For each odd sized block define
\begin{equation}
  \label{eq:epsilon}
\Ep_{2k+1}=i^{k}v\prod (1-e_jf_j).  
\end{equation}
This is an element in $Pin(V),$ and acts by $-Id$ on the block. Even
products of $\pm \Ep_{2k+1}$  belong to $Spin(V),$ and represent the
connected components of $C_{\wti{K}}(e).$   
\begin{prop}
Let $m$ be the number of distinct odd blocks. Then 
$$A_K(\calO) \cong  \begin{cases} \bbZ _2 ^{m-1} & \mbox{ if }  m >  0 \\ 1 & \mbox{ if } m=0.  \end{cases} $$ Furthermore, 
\begin{enumerate}

\item If $E$ has an odd block of size $2k+1$ with $r_{2k+1}>1,$ then
$A_{\wti{K}}(\calO)\cong A_K(\calO).$

\item If all $r_{2k+1}\le 1,$ then there is an exact sequence
$$
1\longrightarrow \{\pm I\}\longrightarrow
A_{\wti{K}}(\calO)\longrightarrow A_K(\calO)\longrightarrow 
0.
$$
\end{enumerate}
\end{prop}
\begin{proof}

Assume that there is an ${r_{2k+1}>1}.$ Let 
$$
\begin{matrix}
&e_1&\rightarrow    &\dots &\rightarrow &e_{2k+1}&\rightarrow 0\\
&f_{2k+1}&\rightarrow&\dots &\rightarrow &-f_1   &\rightarrow 0
\end{matrix}
$$   
be two of the blocks. In the Clifford algebra this element is
$e=(e_2f_1+\dots +e_{2k+1}f_{2k})/2.$ The element ${\sum \limits _{ j=1} ^{2k+1}}(1-e_{j}f_{j})$ in the Lie
algebra commutes with $e$. So its exponential
\begin{equation}\label{path1}
\prod \exp\big( i\theta(1-e_{j}f_{j})/2\big)=
\prod [\cos\theta/2 + i\sin\theta/2 (1-e_{j}f_{j})]
\end{equation}
also commutes with $e.$  {At $\theta=0$, 
the element in (\ref{path1})  is $I$; at $\theta =2\pi$, it is $-I$.} Thus $-I$ is in the connected component of the identity of
$A_{\wti{K}}(\calO)$ (when $r_{2k+1}>1$), and therefore $A_{\wti{K}}(\calO)=A_K(\calO).$  

\medskip
Assume there are no blocks of odd size. Then $C_K(\calO) {\cong \prod Sp (r_{2l})}$ is simply
connected, so $C_{\wti{K}}(\calO)\cong C_K(\calO)\times\{\pm I\}.$  { Therefore $A_{\tu{K}} (\calO) \cong \bbZ _2$.}

\medskip
Assume there are {$m$ distinct odd blocks with $m\in 2\bbZ _{>0}$ and $r_{2k_1+1}=\cdots =r_{2k_m +1}=1.$ In this case, $C_K(\calO)\cong \prod Sp(r_{2l}) \times S[  \underset{m}{\underbrace{O(1) \times \cdots \times O(1)}} ]$
, and hence $A_{\tu{K}} (\calO)\cong \bbZ_2^{m-1}$. 
Even products of $\{ \pm \Ep _{2k_j +1} \}$ are representatives of elements in $A_{\tu{K} }(\calO)$.} They satisfy  

$$
\Ep_{2k+1}\cdot\Ep_{2\ell+1}=
\begin{cases}
-\Ep_{2\ell +1}\cdot\Ep_{2k+1}  &k\ne \ell,\\
{ (-1)^kI}  &k=\ell.
\end{cases}
$$  
\end{proof}

\begin{cor} \ 
\label{c:cgp}
  \begin{enumerate}
  \item If $\calO=[3\ 2^{n-2}\ 1],$ then {$A_{\wti{K}}(\calO)\cong
\bb Z_2\times\bb Z_2=\{\pm\Ep_3\cdot\Ep_1,\pm
I\}$}. 
\item If $\calO=[3\ 2^{2k}\ 1^{2n-4k-3}]$ with  $2n-4k-3>1,$ then  $A_{\wti{K}}(\calO)\cong
\bb Z_2.$
\item If  $\calO=[2^{n}]_{I,II}$ ($n$ even), then $A_{\wti{K}} (\calO) \cong\bb Z_2.$ 
\item If $\calO=[2^{2k}\ 1^{2n-4k}]$ with $2k<n,$ then
  $A_{\wti{K}} (\calO) \cong 1.$ 
  \end{enumerate}
In all cases  $C_{\wti K}(\calO)=Z(\wti K)\cdot C_{\wti K}(\calO)^0.$ 

\end{cor}

\end{document}